\documentclass{article}

\usepackage{amsmath}
\usepackage{amsthm}
\usepackage{amssymb}
\usepackage{mathrsfs}
\usepackage{graphicx}
\usepackage{verbatim}
\usepackage{tikz}
\usepackage[all]{xy}
\usepackage{graphics}
\usepackage[colorlinks=true]{hyperref}
\hypersetup{urlcolor=blue, citecolor=red}

\newtheorem{theorem}{Theorem}[section]
\newtheorem{lemma}[theorem]{Lemma}
\newtheorem{corollary}[theorem]{Corollary}
\newtheorem{proposition}[theorem]{Proposition}
\theoremstyle{definition}

\newtheorem{remark}{Remark}

\newcommand{\R}{\mathbb{R}}

\newcommand{\N}{\mathbb{N}}

\def\beq{\begin{equation}}
\def\eeq{\end{equation}}

\title{\sc Some symmetry results for entire solutions of an elliptic system arising in
phase separation}
\date{}
\author{Alberto Farina }
\begin{document}
\numberwithin{equation}{section}
\maketitle
{\footnotesize
\centerline{LAMFA, CNRS UMR 7352, Universit\'e de Picardie Jules Verne}
\centerline{33 rue Saint-Leu, 80039 Amiens, France}
\centerline{and}
\centerline{Institut Camille Jordan, CNRS UMR 5208, Universit\'e Claude Bernard Lyon I}
\centerline{43 boulevard du 11 novembre 1918, 69622 Villeurbane cedex, France}
\centerline{email: alberto.farina@u-picardie.fr}}

\centerline{}

\centerline{}

\begin{abstract} 
We study the one dimensional symmetry of entire solutions to an elliptic system arising in phase separation  for Bose-Einstein condensates with multiple states. We prove that any monotone solution, with {\em arbitrary algebraic growth} at infinity, must be one dimensional in the case of two spatial variables.  We also prove the one dimensional symmetry for 
{\em half-monotone solutions}, i.e., for solutions having only one monotone component. 

\end{abstract}


\medskip

\section{Introduction and main results}

We study smooth solutions of the elliptic system

\beq\label{Sistema}
\begin{cases}
\Delta u  = u v^2 & \text{in $\R^N$}\\
\Delta v  = v u^2 & \text{in $\R^N$}\\
u, v >0  & \text{in $\R^N$}
\end{cases}
\eeq

\noindent where the pair $(u,v)$ has at most algebraic growth at infinity and $N\ge2$. 

\bigskip

Problems and solutions of this type naturally arise in the study of phase separation phenomena for Bose-Einstein condensates with multiple states (cfr. \cite{BLWZ}, \cite{BTWW} and the references therein).

In order to motivate our study and to understand the difficulties that one has to face when dealing with system \eqref{Sistema}, we review the known results about the considered problem.
The one dimensional case was studied in \cite{BLWZ}. The Authors of  \cite{BLWZ} proved the existence, symmetry, monotonicity and the growth estimates for the solutions to \eqref{Sistema}. In particular they proved that 

\[
\exists \, x_0 \in \R  \quad  : \quad
u(x-x_0) = v(x_0-x) \qquad \forall \, x \in \R,
\]

\[
{\rm either } \quad u' >0, \, v'<0 \quad {\rm or }\quad u'<0, \, v'>0
\]
and
\[
u(x) +v(x) \le C (1 + \vert x \vert) \qquad \forall \, x \in \R.
\] 
Uniqueness (up to translations, scaling and reflection) of the solution  to the one dimensional system \eqref{Sistema} has been recently settled in \cite{BTWW}. Thus, the one dimensional case is well-understood. On the other hand, the higher dimensional case is more involved and much less is known in that situation.   

In \cite{NTTV} it is proved that, in any dimension $N\ge1,$ there are no solutions to \eqref{Sistema} with sublinear growth, i.e., such that 

\[
 \exists \quad \alpha \in (0,1) \quad : \quad
u(x)+v(x) \le C(1 + \vert x \vert)^{\alpha} \qquad \forall \,x \in \R^N.
\]

On the other hand, when $N=2$, solutions with arbitrary integer algebraic growth at infinity has been recently constructed in \cite{BTWW}.  In particular, these solutions are not one dimensional when the growth at infinity is superlinear, showing thus the great difference between the one dimensional case (where all the solutions have linear growth) and the higher dimensional case. 

Inspired by a celebrated conjecture of E. De Giorgi \cite{DeG} about monotone solutions to the Allen-Cahn equation (see also \cite{FV} for a recent review on the conjecture of De Giorgi and related topics) and motivated by the results in the one dimensional case, Berestycki, Lin, Wei and Zhao \cite{BLWZ} raised the following 

\medskip

\noindent {\textbf {Question (\cite{BLWZ})} {\em Let $ N>1$. Under what conditions is it true that all monotone solutions to \eqref{Sistema}, i.e., such that 

\beq \label{def:mono}
\frac{\partial u}{\partial x_N} >0, \quad  \frac{\partial v}{\partial x_N} <0 \quad in \,\, \R^N,
\eeq

\noindent are one dimensional ? (That is, there exist $ U,V : \R \rightarrow \R$ and a unit vector $\nu$ such that $(u(x),v(x)) = (U(\nu\cdot x), V(\nu\cdot x))$  for every $ x \in \R^N $ ? )}

\medskip

They gave a positive answer to the above question if

\beq \label{monoBLWZ}
N=2 \qquad {and} \qquad  
u(x) +v(x) \le C (1 + \vert x \vert) \qquad \forall \, x \in \R^2.
\eeq

\noindent Both the assumptions are crucial in their proof. In particular, in their approach it is not possible to replace the growth condition in \eqref{monoBLWZ} by the more general condition 

\[
u(x)+v(x) \le C(1 + \vert x \vert)^{1+ \epsilon} \qquad \forall \,x \in \R^2,
\]
for some $\epsilon>0$. 

\medskip

Our first result states that, for $N=2$, monotone solutions to \eqref{Sistema}  with at most {\em arbitrary algebraic growth} at infinity, must be one dimensional. 

\medskip 

\begin{theorem}\label{teo monot}
Let $N=2$. Then any {\em monotone} solution $(u,v)$ to \eqref{Sistema} with at most {\em arbitrary} algebraic growth at infinity, must be one dimensional. 
\end{theorem}
Our proof uses a different strategy based on the {\it Almgren frequency function} 

\beq \label{Almgren}
{\cal N}(r) : = \frac{r \int_{B_r(0)} \vert \nabla u \vert^2 + \vert \nabla v \vert^2 + u^2v^2}{\int_{\partial B_r(0)} u^2 + v^2}, \qquad r>0.
\eeq
We shall describe it in Section 3. 

\medskip

To state our second result we need the following 

\medskip

\noindent {\textbf {Definition (Half-monotone solution)} {\em A solution $(u,v)$ to \eqref{Sistema} is said to be {\em half-monotone} if it  has {\em one monotone component} (that is,  if either 
$ \frac{\partial u}{\partial x_N} >0$ or $\frac{\partial v}{\partial x_N} <0$ in $\R^N$).}

\medskip

Our second result states that, for $N=2$, half-monotone solutions to \eqref{Sistema}  with at most {\em arbitrary algebraic growth} at infinity, must be one dimensional. 

\medskip 

\begin{theorem}\label{teo half-monot}
Let $N=2$. Then any {\em half-monotone} solution $(u,v)$ to \eqref{Sistema} with at most {\em arbitrary algebraic growth} at infinity, must be one dimensional. 
\end{theorem}

\medskip

Theorem \ref{teo monot} is proved in Section 3, while Section 4 will be devoted to Theorem \ref{teo half-monot}.

\medskip

\section{Some auxiliary results}

\medskip

In this section we prove some preliminary results which will be used in the course of the main theorems. 

We first recall that the {\it Almgren frequency function}  defined by \eqref{Almgren} is {\it nondecreasing} in $r$ (cfr. Proposition 5.2  of \cite{BTWW}) and then we prove the following result. 

\begin{lemma} \label{lemma:crescita alg}
Assume $N \ge 1$ and let $(u,v)$ be a solution to \eqref{Sistema} with algebraic growth at infinity, i.e., satisfying 

\beq \label{crescita}
\exists \quad \alpha \ge 1 \quad : \quad u(x)+v(x) \le C(1 + \vert x \vert)^{\alpha} \qquad \forall \,x \in \R^N.
\eeq
Then 
\beq \label{Ninfinito}
{\cal N}(\infty) : = \lim_{r \to +\infty} {\cal N}(r) \le \alpha.
\eeq
\end{lemma}

\begin{proof}

For $ r>0$ we set $ H(r) : =  r^{1-N}\int_{\partial B_r(0)} u^2 + v^2$ and $q(r) = \frac{H(r)}{r^{2{\cal N}(r_0)}} $.  A direct computation gives that 

\beq \label{q-monot}
\forall \, r_0>0 \qquad r \rightarrow q(r) \qquad is \,\, nondecreasing \,\,  for \quad r>r_0.
\eeq
Indeed, a direct calculation yields $H'(r) = 2 r^{1-N} \int_{B_r(0)} \vert \nabla u \vert^2 + \vert \nabla v \vert^2 + 2u^2v^2$ (cfr. Section 5 of \cite{BTWW}) and so $ q'(r) = 2 r^{-2{\cal N}(r_0)} [ r^{-1}{\cal N}(r)H(r) - r^{-1} {\cal N}(r_0) H(r) + \int_{B_r(0)} u^2v^2 ] \ge 2 r^{-2{\cal N}(r_0)} [ r^{-1}{\cal N}(r)H(r) - r^{-1} {\cal N}(r_0) H(r) ] \ge 0$ for any $ r>r_0$, where in the latter we have used the monotonicity of the Almgren frequency function ${\cal N}$.  

From \eqref{q-monot} we infer that 

\beq \label{stima1:crescita}
\forall \, r_0>0 \quad \exists \, c(r_0)>0 \quad : \quad c(r_0) r^{2{\cal N}(r_0)} \le H(r) \qquad \forall \, r>r_0
\eeq
and thus, by \eqref{crescita}, 

\beq \label{stima2:crescita}
\forall \, r_0>0 \quad \exists \, c(r_0)>0 \quad : \quad c(r_0) r^{2{\cal N}(r_0)} \le c_1 r^{2\alpha} \qquad \forall \, r>r_0
\eeq
where $ c_1 $ is a positive constant depending only on the dimension $N$ and on the constant $C$ appearing in \eqref{crescita}.  From \eqref{stima2:crescita} we immediately get 

\beq \label{stima3:crescita}
\forall r_0>0 \qquad {\cal N}(r_0) \le \alpha
\eeq
and the desired conclusion \eqref{Ninfinito} follows from the monotonicity of the Almgren frequency function ${\cal N}$. 
\end{proof}

Now we prove a Liouville-type theorem which will be useful in Section 4.

\medskip

\begin{proposition}  \label{Liouville} 
Assume $N \ge 1$. Let $v,  \sigma \in C^2(\R^N)$  be functions satisfying $v>0$ on $\R^N$, 
\beq \label{equaLiouvi}
- div(v^2 \nabla \sigma) \le 0 \quad in \,\, \R^N
\eeq
and
\beq \label{hypoLiouv}
\int_{B_R(0)} ( v \sigma^+)^2 \le C R^2 \qquad \forall \, R>>1,
\eeq
for some positive constant $C$ independent of $R$.

\noindent Then $ \sigma^+ = const.$
\end{proposition} 

\begin{proof} Let $\phi \in C^{\infty}_c (\R^N)$ be a function such that $ 0 \le \phi \le 1$, and

\beq\label{cut-off}
\phi(x) := \begin{cases}
1\quad & {\rm if} \qquad \vert x \vert \le 1,\\
0 & {\rm if} \qquad \vert x \vert \ge 2.
\end{cases}
\eeq
For $ R>1$ and $x \in \R^N$, let $ \phi_R(x) = \phi({x \over {R}})$.
Multiplying \eqref{equaLiouvi} by $\sigma^+ \phi^2_R$ and integrating by parts, we find

\[
\int \phi^2_R v^2 \vert \nabla \sigma^+ \vert^2 \le - 2 \int \phi_R v^2 \sigma^+ (\nabla \phi_R  \cdot \nabla \sigma) \le
\]
\[
\le C_1 \Big [ \int_{\{R \le \vert x \vert \le 2R\}}  \phi^2_R v^2 \vert \nabla \sigma^+ \vert^2 \Big ]^{1 \over 2}
\Big [ {1 \over {R^2}} \int_{\{\vert x \vert \le R\}}  (v \sigma^+)^2  \Big ]^{1 \over 2}, 
\]
\noindent for some positive constant $C_1$ independent of $R$. Now, the assumption \eqref{hypoLiouv}  yields :

\beq \label{disugLiouv}
\int \phi^2_R v^2  \vert \nabla \sigma^+ \vert^2 \le C_1 C^{1 \over 2} \Big [ \int_{\{R \le \vert x \vert \le 2R\}}  \phi^2_R v^2 \vert \nabla \sigma^+ \vert^2 \Big ]^{1 \over 2}\qquad {\rm for} \quad R>>1,
\eeq
which implies $ v^2 \vert \nabla \sigma^+ \vert^2  \in L^1 ({\R^N}).$ Using the latter information in \eqref{disugLiouv} and letting $R \rightarrow +\infty$, we obtain $ v^2 \vert \nabla \sigma^+ \vert^2 \equiv 0$, which implies $ \sigma^+ = const.$ 
\end{proof}

We close the present section by recalling a result proved in \cite{BTWW} (cfr. Theorem 1.4. therein). 

\begin{theorem} [\cite{BTWW}]\label{teo:sus}
Assume $ N\ge2$. Let $(u,v)$ be a solution to \eqref{Sistema} such that ${\cal N}(\infty)$ is finite. Then
\beq \label{intero}
{\cal N}(\infty) = d \in \N^{\star}
\eeq
and
there is a homogeneous harmonic polynomial of degree $d$, denoted by $\Psi$, such that the blow-down sequence defined by :
\beq \label{}
(u_R(x), v_R(x)) : = \Big( \frac{1}{L(R)} u(Rx), \frac{1}{L(R)} v(Rx) \Big), \qquad R>0,
\eeq
\beq \label{}
where \qquad  L(R)>0 \quad  : \quad \int_{\partial B_1} u_R^2 + v_R^2 = 1,
\eeq
converges (up to a subsequence) to $(\Psi^+, \Psi^-)$ uniformly on compact sets of $\R^N.$ 
In addition, if ${\cal N}(\infty) =1$ then $(u,v)$ has linear growth at infinity. 
\end{theorem}

\medskip

Here, and in the sequel, we denote by $w^+$ the positive part of the function $w$ and by $w^-$ the negative part of $w$.  

\medskip

\section{Monotone solutions}

\medskip

\begin{proof}[Proof of Theorem \ref{teo monot}]  By Lemma \ref{lemma:crescita alg} we get that ${\cal N}(\infty)$ is finite.  This enables us  to use Theorem \ref{teo:sus} with $N=2$.  The monotonicity assumption implies that $ \frac{\partial u_R}{\partial x_2} >0$ and $\frac{\partial v_R}{\partial x_2} <0 $ in $\R^2$, for every $R>0$. Therefore, for every $x=(x_1,x_2) \in \R^2$, every $ t>0$ and every $ R>0$ we have 
\beq \label{monot:discr}
u_R(x_1,x_2 + t) \ge u_R(x_1,x_2), \qquad  v_R(x_1,x_2 + t) \le v_R(x_1,x_2), 
\eeq

\noindent and an application of Theorem \ref{teo:sus} immediately yields that $\Psi^+$ is nondecreasing with respect to $x_2$, while $\Psi^-$ is nonincreasing with respect to $x_2$.  In particular we obtain that $\frac{\partial \Psi}{\partial {x_2}} \ge 0 $ in $\R^2$. 

To conclude the proof we invoke the subsequent Proposition \ref{armoniche:monot}, which tells us that $\Psi$ must be a linear function. Hence, $ d=1$ in \eqref{intero}, which means that $(u,v)$ has at most linear growth at infinity, that  is \eqref{monoBLWZ} is satisfied. The desired result then follows from \cite{BLWZ}, as discussed in the Introduction. \end{proof}

\smallskip

Now we turn to Proposition \ref{armoniche:monot}, which deals with entire monotone harmonic functions.

\begin{proposition} \label{armoniche:monot}
Assume $N\ge 2$ and let $H$ be a harmonic function on $\R^N$ such that 
\[
\frac{\partial H}{\partial x_N} \ge 0 \quad in \,\, \R^N.
\]
Then 
\beq \label{decompArmonica}
H(x) = \gamma x_N + h(x_1,...,x_{N-1}) \qquad \forall \, x \in \R^N,
\eeq
where $h$ is a harmonic function on $\R^{N-1}$ and $ \gamma \in \R$.

\noindent In particular, $H$ must be an affine function when $N=2$.

\end{proposition}

\begin{proof}[Proof of Proposition \ref{armoniche:monot}] 
By assumption we have that $\frac{\partial H}{\partial {x_N}}$ is a {\em nonnegative} entire harmonic function. So, it must be constant by the classical Liouville Theorem, say $\frac{\partial H}{\partial {x_N}} \equiv \gamma \in \R $.  
In particular, the function $h := H - \gamma x_N$ satisfies $\frac{\partial h}{\partial {x_N}} \equiv 0$ and thus, it must be an entire harmonic function depending only on the variables $x_1,...,x_{N-1}$. This gives \eqref{decompArmonica}. When $N=2$, $h$ must be affine (since in this case $h$ depends only on one variable). This concludes the proof. 
\end{proof}

\begin{corollary} \label{oss:armo-monot}
Assume $N\ge 3$ and let $H$ be a homogeneous harmonic polynomial of degree $d \ge 1$  such that 
\[
\frac{\partial H}{\partial x_N} \ge 0 \quad in \,\, \R^N.
\]
Then we have the following alternative :
\begin{enumerate}
\item[($i$)] either $H$ is a linear  function 
\item[($ii$)] or $H$ is a homogeneous harmonic polynomial of degree $d \ge 2$ in the first $N-1$ variables. 
\end{enumerate}
\end{corollary}

\begin{proof}
If $d=1$, $H$ is cleary linear. If $d\ge2$, then $\gamma =0$ in \eqref{decompArmonica}, since $H$ is also a homogeneous function of degree $d$. This proves the corollary. 
\end{proof}

\smallskip

To conclude the section we prove a proposition which will be crucial in the proof of Theorem \ref{teo half-monot}.

\smallskip

\begin{proposition} \label{polArm+ : monot} 
Assume $N=2$ and let $P$ be a homogeneous harmonic polynomial of degree $d \ge 1$ such that $P^+$ is nondecreasing with respect to $x_2$. 

\noindent Then $P$ must be a linear function. 
\end{proposition}

\begin{remark}  
{\em In the above proposition the homogeneous harmonic polynomial $P$ cannot be replaced by an arbitrary harmonic function as shown by $H(x,y) = e^y \sin(x) $  in $\R^2$. } 
\end{remark}

\begin{proof}[Proof of Proposition \ref{polArm+ : monot} ]

If $P$ vanishes at a point $z_0 = (x_0,y_0) \in \R^2 \setminus \{ 0 \}$, it must vanishes on the entire straight line passing through $z_0$ and the origin. To see this, we first observe that $P$ vanishes on the half-line 
$\{ t z_0 \, : \, t \ge 0 \}$ by homogeneity. On the other hand, the restriction of $P$ to the entire straight line passing through $z_0$ and the origin is a polynomial of one variable, which is identically zero on the half-line $\{ tz_0 \, : \, t\ge 0\}$. This clearly implies that $P$ must vanish identically on the entire straight line containing $\{ t z_0 \, : \, t \ge 0 \}$. 

Since $ d\ge 1$, the polynomial $P$ must vanish somewhere outside the origin. 

Suppose that $P(z_0) = 0$ for some $z_0 = (x_0,y_0)$ with $ x_0 \neq  0$ and denote by ${\cal S}_{z_0}$ the straight line passing through $z_0$ and the origin. By the monotonicity assumption on $P^+$ we get that $ P \le 0$ on the open half-plane lying below  ${\cal S}_{z_0}$. Indeed, if $ P(x,y) >0$, then $P(x,s)>0$ for every $ s\ge y$, since  $P^+$ is nondecreasing on $\R^2$. 

Now, since $P$ is harmonic and nonconstant, the strong maximum principle implies that $P < 0$ everywhere on  the open half-plane lying below  ${\cal S}_{z_0}.$ Thus, an application of Hopf's Lemma gives that $ \vert \nabla P \vert >0$ on the straight line ${\cal S}_{z_0}$. In particular $ \vert \nabla P (0)\vert >0$, which clearly implies $d=1$ (by the  homogeneity of $P$) and $ P(x,y) = \alpha x +\beta y $, with $\beta >0$. 

Next we suppose that $P(z_0) = 0$ for some $z_0 = (0,y_0)$ and $ y_0 \neq 0$. In this case $P$ vanishes on the straight line $ \{ (0,y) \, : \, y \in \R \}$ and the above argument  tell us that $P$ cannot vanish on $ \R^2 \setminus \{ (0,y) \, : \, y \in \R \}$. Hence, either $P>0$ or $P<0$ on the open half-plane  
$ \{ (x,y) \in \R^2  \, : \, y>0  \}$. Applying once again Hopf's Lemma we get $ d=1$ and then $ P(x,y) = \alpha x$. This concludes the proof. 
\end{proof}

\medskip

\section{Half-monotone solutions}

\medskip

\begin{proof}[Proof of Theorem \ref{teo half-monot}]  Without loss generality we can suppose that $ \frac{\partial u}{\partial x_2} >0$ in $\R^2$. If we prove that $ \frac{\partial v}{\partial x_2} < 0$ in $\R^2$ we are done, since in this case the desired conclusion will follow from Theorem \ref{teo monot}. To this end, we first prove that $(u,v)$ has at most linear growth at infinity. 

\smallskip

By Lemma \ref{lemma:crescita alg} we get that ${\cal N}(\infty)$ is finite and so we can use Theorem \ref{teo:sus} with $N=2$.  Since $(u,v)$ is half-monotone we see that $ \frac{\partial u_R}{\partial x_2} >0$  in $\R^2$, for every $R>0$. Therefore, for every $x=(x_1,x_2) \in \R^2$, every $ t>0$ and every $ R>0$ we have 
\beq \label{half-monot:discr}
u_R(x_1,x_2 + t) \ge u_R(x_1,x_2)
\eeq

\noindent and an application of Theorem \ref{teo:sus} immediately yields that $\Psi^+$ is nondecreasing with respect to $x_2$. The latter enables us to invoke Proposition \ref{polArm+ : monot} from which we infer that $\Psi$ is linear. Hence $d=1$ and $(u,v)$ has at most linear growth at infinity.  With this information in our hands we are ready to prove 
that $ \frac{\partial v}{\partial x_2} < 0$ in $\R^2$. The latter claim follows from the next result. This completes the proof of Theorem \ref{teo half-monot}.
\end{proof}

\begin{theorem} Let $(u,v)$ be a solution of

\beq 
\begin{cases}
\Delta u  = u v^2 & \text{in $\R^2$}\\
\Delta v  = v u^2 & \text{in $\R^2$}\\
u, v >0  & \text{in $\R^2$}
\end{cases}
\eeq
such that 
\beq \label{crescLin}
u(x) + v(x) \le C (1 + \vert x \vert) \qquad \forall \, x \in \R^2,
\eeq
\beq \label{unaderivata}
\frac{\partial u}{\partial x_2} >0 \quad in \,\, \R^2.
\eeq
Then 
\[
\frac{\partial v}{\partial x_2} <0 \quad in \,\, \R^2.
\]
\end{theorem}

\begin{proof} Set $ u_2 = \frac{\partial u}{\partial x_2}$ and $v_2 = \frac{\partial v}{\partial x_2} $, then differentiating the second equation in \eqref{Sistema} we get 
\beq \label{primav2} 
\Delta v_2 = v_2 u^2 + 2vuu_2 > v_2u^2 \quad in \,\, \R^2.
\eeq
Here we have used \eqref{unaderivata} and $u, v >0$ in $\R^2$. 

On the other hand, if we set $ \sigma = \frac{v_2}{v}$ a direct calculation gives 
\[
\Delta v_2 = \frac{div(v^2 \nabla \sigma)}{v} +  v_2 u^2\quad in \,\, \R^2.
\]
Thus,  
\beq \label{equaLiouv}
- div(v^2 \nabla \sigma) \le 0 \quad in \,\, \R^2.
\eeq

Testing the second equation in \eqref{Sistema} with $v \phi_r^2$, where $\phi_r$ is the standard cut-off function defined in the proof of Proposition \ref{Liouville}, we have :
\beq \label{stima}
\int_{B_r} \vert \nabla v \vert^2 \le C r^2 \qquad \forall \, r>0 
\eeq
where $C$ is a positive constant independent of $r$. Note that in the latter estimate we have used in a crucial way the linear growth of $v$, i.e., the assumption \eqref{crescLin}. 

We observe that 

\[
0 \le v \sigma^+ = v \Big( \frac{v_2}{v} \Big)^+ = v_2^+ \le \vert \nabla v \vert \qquad {\rm on} \quad \R^2
\]
together with \eqref{equaLiouv} and \eqref{stima}, enables us to apply Proposition \ref{Liouville} to infer that 
$ \sigma^+ = const. = \lambda \ge 0$.  We claim that $ \lambda =0$. Indeed, $\lambda >0$ implies 

\beq \label{ultima}
v_2 = \lambda v >0 \qquad {\rm and} \qquad \Delta v_2 = \lambda \Delta v = \lambda v u^2 = v_2 u^2,
\eeq
which is in contradiction with \eqref{primav2}. Hence, $ \sigma^+ \equiv \lambda = 0$ and so $ v_2 \le 0 $ on $\R^2$. The strong maximum principle applied to \eqref{primav2} then gives $ v_2 < 0$ on $\R^2$. 
\end{proof}

\vspace{0.5cm}

\noindent \textbf{Acknowledgements: } The results contained in this article were presented at the Workshop {\it  Singular limit problems in nonlinear PDEs} on November 2012 at CIRM Luminy (France). The author wishes to thank the organizers for the invitation and their kind hospitality. The author is supported by the ERC grant EPSILON ({\it Elliptic Pde's and Symmetry of Interfaces and Layers for Odd Nonlinearities}).


\vspace{0.5cm}

\begin{thebibliography}{99}

\bibitem{BLWZ}
\newblock H. Berestycki, T-C. Lin, J. Wei, C. Zhao, 
\newblock \emph {On Phase-Separation Model: Asymptotics and Qualitative
Properties}, 
\newblock preprint 2009. To appear in Archive for Rational Mechanics and Analysis. 


\bibitem{BTWW}
\newblock H. Berestycki, S. Terracini, K.  Wang,  J. Wei,
\newblock \emph {On Entire Solutions of an Elliptic System Modeling Phase Separation},
\newblock preprint 2012,  arXiv:1204.1038 


\bibitem {DeG}
\newblock E. De Giorgi,
\newblock \emph {Convergence problems for functionals and operators}. 
\newblock Proceedings of the International Meeting on Recent Methods in Nonlinear Analysis (Rome, 1978). Pitagora, Bologna, 1979, pp. 131-188.



\bibitem{FV}
\newblock A. Farina, E. Valdinoci,
\newblock \emph {The state of the art for a conjecture of De Giorgi and related problems},
\newblock Recent Progress on Reaction-Diffusion Systems and Viscosity Solutions. World
Scientific Publishers, Hackensack, NJ, 2009, pp. 74-96.




\bibitem{NTTV}
\newblock B. Noris, H. Tavares, S. Terracini, G. Verzini, 
\newblock \emph {Uniform Holder bounds for nonlinear Schrodinger
systems with strong competition}, 
\newblock Comm. Pure Appl. Math. 63 (2010), 267-302. 


\end{thebibliography}
\end{document}